\documentclass[20pt, oneside]{article}   	
\usepackage{geometry}                		
\usepackage{amsmath}
\usepackage{amsthm}	
\usepackage{amssymb}
\usepackage{tikz}
\usepackage[implicit]{hyperref}
\usepackage[bottom]{footmisc}
\usepackage[affil-it]{authblk}
\usepackage[toc,page]{appendix}
\newtheorem{thm}{Theorem}[section]

\newtheorem{prop}[thm]{Proposition}

\newtheorem{lem}[thm]{Lemma}
\newtheorem{cor}[thm]{Corollary}
\newtheorem*{claim*}{Claim}

\newtheorem*{thm*}{Theorem}

\newtheorem{thmm}{Main Theorem}

\newtheorem{ex}[thm]{Example}

\theoremstyle{definition}
\newtheorem*{rem*}{Remark}
\newtheorem{defi}[thm]{Definition}
\newtheorem{rem}[thm]{Remark}

\newcommand{\F}{\mathbb{F}}
\newcommand{\fl}{\mathfrak{l}}

\newcommand{\fp}{\mathfrak{p}}
\newcommand{\fP}{\mathfrak{P}}
\newcommand{\Gal}{\mathrm{Gal}}
\newcommand{\Aut}{\mathrm{Aut}}

\newcommand{\GL}{\mathrm{GL}}
\newcommand{\cO}{\mathcal{O}}
\newcommand{\M}{\mathcal{M}}

\newcommand{\Fl}{F_{\mathfrak{l}}}
\newcommand{\Al}{A_{\mathfrak{l}}}
\newcommand{\Fq}{\mathbb{F}_q}
\newcommand{\Ksep}{K^{\mathrm{sep}}}
\newcommand{\Kbar}{\overline{K}}
\newcommand{\Ebar}{\overline{E}}
\newcommand{\GK}{G_K}
\newcommand{\End}{\operatorname{End}}
\newcommand{\Hom}{\operatorname{Hom}}
\newcommand{\SL}{\operatorname{SL}}
\newcommand{\tr}{\operatorname{tr}}
\newcommand{\Tr}{\operatorname{Tr}}
\newcommand{\Frob}{\operatorname{Frob}}

\newcommand{\udens}{\overline{\delta}}

\newcommand{\xdownarrow}[1]{%
  {\left\downarrow\vbox to #1{}\right.\kern-\nulldelimiterspace}
}

\title{Frobenius Traces for Drinfeld Modules and a Strong Multiplicity
One Theorem in Positive Characteristic}
\author{Chien-Hua Chen \thanks{Electronic address: \texttt{cc45@aub.edu.lb}; ORCID: \texttt{0000-0003-3267-5603} ; Corresponding author}}
\affil{Department of Mathematics, American University of Beirut}

\begin{document}

\maketitle
\abstract

We study to what extent Frobenius traces determine continuous Galois
representations of a global function field $K$ over a local field $E$ of
positive characteristic. We show that the principle ``Frobenius trace
equality at all but finitely many places forces isomorphism'' holds
whenever one of the two representations is absolutely irreducible, and the other is irreducible.
However, the principle fails for irreducible representations in general. As an
application, Frobenius traces determine non-CM Drinfeld modules of
arbitrary rank in generic characteristic up to isogeny, while CM Drinfeld
modules with inseparable CM-field can give counterexamples.
Finally, by comparing Frobenius characteristic polynomials, we establish a
positive characteristic analogue of the strong multiplicity one principle
under a thinness condition on the Zariski closure of the image. In particular, for pairs of non-CM Drinfeld modules, assuming an open image
property of the product representation, we prove a function field analogue
of Rajan's result on twists of elliptic curves.\section {Introduction}

Brauer and Nesbitt proved a theorem which implies that two semisimple Galois representations with image in $\GL_n$ over a field of characteristic $0$ are isomorphic if and only if the two representations have the same Frobenius trace for all but finitely many places. As an application to representations associated to elliptic curves, we can conclude that Frobenius traces can distinguish elliptic curves up to isogeny. In the positive characteristic setting, however, this principle becomes more subtle. The classical argument relies on Newton identities to reconstruct characteristic polynomials from traces, which relies on divisibility conditions that may fail in characteristic $p$. As a result, Frobenius traces alone do not, in general, determine representations up to isomorphism. Understanding when such a determination remains possible then becomes a natural problem.

Our first main result shows that, in contrast with the general situation, Frobenius traces do determine Drinfeld modules without complex multiplication up to isogeny, in every rank and every field characteristic $p$ of $\F_q$. For a global function field $K$, we denote $G_K=\Gal(K^{\rm sep}/K)$.

\begin{thmm}[Corollary~\ref{corr}]
Let $q=p^e$, $A=\F_q[T]$ with field of fractions $F=\F_q(T)$, and let
$K$ be a finite extension of $F$. Let $\phi_1,\phi_2$ be Drinfeld
$A$-modules of rank $r$ over $K$ of generic characteristic with
$\End_{\Kbar}(\phi_i)=A$. Let $\fl$ be a prime of $A$, let $\Fl$ be the
completion of $F$ at $\fl$, and let
$\rho_{i,\fl}\colon \GK\to\GL_r(\F_l)$ be the $\fl$-adic representations.
If $\tr\rho_{1,\fl}(\Frob_v)=\tr\rho_{2,\fl}(\Frob_v)$ for all but finitely
many places $v$ of $K$, then $\phi_1$ and $\phi_2$ are $K$-isogenous. In particular $\rho_{1,\fl}\cong\rho_{2,\fl}$.
\end{thmm} 

For $p>r$ this follows from the Newton identities, thus the point is the case
$p\le r$. For $r=2$ we also give a second proof (Theorem~3.3), based on
the Weil pairing, which avoids Burnside's theorem. Main Theorem~1 is
deduced from the following criterion.

\begin{thmm}[Proposition \ref{Main1}]

Let $K$ be a global function field, $G_K=\Gal(K^{\rm sep}/K)$, and $E$ be any local field of characteristic $p>0$. Let $n\in \mathbb{N}$, and $\rho_1, \rho_2:G_K\rightarrow \GL_n(E)$ be continuous representations, unramified outside a finite set $S$ of places of $K$. Suppose further that
\begin{itemize}
\item For any $ \fp\not\in S$, we have ${\rm Tr}(\rho_1({\rm Frob}_\fp))={\rm Tr}(\rho_2({\rm Frob}_\fp))$.
\item $\rho_1$ is absolutely irreducible and $\rho_2$ is irreducible.
\end{itemize}
Then $\rho_1\cong\rho_2$.
\end{thmm}

The absolute irreducibility of $\rho_1$ cannot be weakened to
irreducibility of both representations (see Remark~\ref{irred}). This
phenomenon occurs for Drinfeld modules with complex multiplication, in
Example~\ref{ex:CM} we construct two CM Drinfeld modules of rank $p$
whose Frobenius traces all vanish but which are not isogenous.

Beyond the Brauer--Nesbitt theorem, Rajan \cite[Thms.~1 and~2]{Ra98} described
what happens when two semisimple $\ell$-adic Galois representations over a
field of characteristic $0$ have equal Frobenius traces on a set of places of
positive upper density. They become isomorphic after restriction to a finite
extension and, under additional hypotheses, differ by a twist by a character
of finite order. In positive characteristic, traces do not determine a
semisimple representation, so we compare Frobenius characteristic polynomials
instead. Our approach combines an algebraic form of the Chebotarev density
theorem (Theorem~\ref{thma}), which computes the density of such places from
the components of the Zariski closure of the image, with a componentwise
thinness condition (\ref{CT}). Theorem~\ref{thmb} gives the corresponding criteria for isomorphism after a
finite extension and for twisting. This theorem applies to any pair of continuous
semisimple representations $\rho_1,\rho_2\colon G_K\to\GL_n(E)$, where $E$ is
a local field of positive characteristic. For
Drinfeld modules, we prove the following analogue of \cite[Cor.~2]{Ra98}.

\begin{thmm}[Proposition~\ref{prop:pairs} and Corollary~\ref{cor:pairs}]
Let $\phi_1,\phi_2$ be non-CM Drinfeld modules of rank $r\ge2$ over $K$ of
generic characteristic, and let $\Sigma_X$ be the set of places at which
their Frobenius characteristic polynomials agree. Put
$\mathbf H=\{(g_1,g_2)\in\GL_r\times\GL_r:\det g_1=\det g_2\}$.
\begin{enumerate}
\item[(a)] If $\phi_1,\phi_2$ are isogenous over $\Kbar$, then
$\rho_{\phi_2,\fl}\cong\rho_{\phi_1,\fl}\otimes\kappa$ for all $\fl$, for a
character $\kappa\colon\GK\to\Fq^\times$ of order $m\mid q-1$, and the Dirichlet density
$\delta(\Sigma_X)=1/m$. Moreover, $m=1$ if and only if $\phi_1,\phi_2$ are
$K$-isogenous.
\item[(b)] If $\phi_1,\phi_2$ are not isogenous over $\Kbar$ and, for some
prime $\fl$, the image of $(\rho_{\phi_1,\fl},\rho_{\phi_2,\fl})$ contains an
open subgroup of $\mathbf H(\Fl)$, then the Dirichlet density $\delta(\Sigma_X)=0$.
\end{enumerate}
In particular, if $\phi_1,\phi_2$ are as in (a) or (b) and
$\delta(\Sigma_X)>1/2$, then $\phi_1,\phi_2$ are $K$-isogenous. The bound
$1/2$ is sharp for $q$ odd.
\end{thmm}

The openness hypothesis in~(b) is the function field analogue of Serre's
open image theorem for products of non-isogenous elliptic curves. By the
recent work of Duan and Fang~\cite[Thm.~1.2]{DF26}, the openness assumption holds for all pairs of
non-CM Drinfeld modules of the same rank that are not isogenous over $\Kbar$,
so~(b) holds unconditionally, see Remark~\ref{uncon}(2) for details. For general
representations, under (\ref{CT}) a density of places with equal
characteristic polynomials greater than $1-1/|\Phi|$ already forces
$\rho_1\cong\rho_2$ (Corollary~\ref{cor1}), and Remark~\ref{rem1} explains why
Rajan's argument deducing $H^\circ\subset X$ from a positive upper density
does not carry over to positive characteristic.

\section*{Acknowledgement}

The main question of this paper was suggested to the author by Professor
Mihran Papikian, whom the author thanks for helpful and inspiring discussions.
The author thanks the referee for a careful reading of the manuscript and for
many valuable suggestions, and Professor Jiangxue Fang for inspiring
discussions on their work \cite{DF26}. The author was partially supported by
the CAMS Fellowship for the 2026--2027 cycle.

\section{Preliminaries}

\subsection{Brauer-Nesbitt Theorem and its related results}

\begin{thm}[Brauer-Nesbitt]\label{BN}

Let $G$ be a group and $E$ be a field. If $\rho_1, \rho_2: G\rightarrow {\rm GL}_n(E)$ are two finite dimensional semisimple representations such that the characteristic polynomials of $\rho_1(g)$ and $\rho_2(g)$ coincide for all $g\in G$, then $\rho_1$ and $\rho_2$ are isomorphic representations. 

\end{thm}

Note that the Brauer-Nesbitt theorem is applied to representations in arbitrary characteristic. However, the following well known application is only restricted to the characteristic $0$ case.

\begin{thm}[ \cite{Lan84}, Corollary 3.8]\label{0case}
Let $G$ be a group and $E$ be a field of characteristic $0$. If $\rho_1, \rho_2: G\rightarrow {\rm GL}_n(E)$ are two finite dimensional semisimple representations such that the traces ${\rm tr}(\rho_1(g))$ and ${\rm tr}(\rho_2(g))$ coincide for all $g\in G$, then $\rho_1$ and $\rho_2$ are isomorphic representations. 

\end{thm}
In the positive characteristic setting, the statement above holds under an additional assumption that ${\rm char}(E)>n$:

\begin{prop}\label{pcase}
Let $G$ be a group, $E$ be a field of characteristic $p$ with $p>n$. If $\rho_1, \rho_2: G\rightarrow {\rm GL}_n(E)$ are two finite dimensional semisimple representations such that the traces ${\rm tr}(\rho_1(g))$ and ${\rm tr}(\rho_2(g))$ coincide for all $g\in G$, then $\rho_1$ and $\rho_2$ are isomorphic representations. \end{prop}

\begin{proof}[Sketch of proof]
The key strategy of the proof of Proposition \ref{pcase}  is to use Newton identities to deduce that one can rebuild the characteristic polynomial of $\rho(g)$ from the trace of the powers ${\rm tr}(\rho(g)^k)$. Suppose $\rho_1(g)$ has eigenvalues $\alpha_1,\cdots,\alpha_n$. Set power sums $p_k={\rm tr}(\rho_1(g)^k)$ for $k\geqslant 1$, and symmetric sums $e_k=\Sigma_{1\leqslant i_1<\cdots< i_k\leqslant n}\alpha_{i_1}\cdots \alpha_{i_k}$ with $e_0=1$. Let us write down the characteristic polynomial of $\rho_1(g)$ as $P_v(X)=X^n-e_1X^{n-1}+e_2X^{n-2}+\cdots+(-1)^ne_n$.
Since ${\rm char}(E)=p>n$, we have that $k$ is invertible in $E$ for $1\leqslant k\leqslant n$. From the Newton identities $$ke_k=\Sigma_{i=1}^k(-1)^{i-1}e_{k-i}p_i \textrm{ for $1\leqslant k\leqslant n$ with $e_0=1$},$$ one can conclude that each $e_k$ is completely determined by $p_1, p_2,\cdots, p_k$, hence by ${\rm tr}(\rho_1(g^i))$ for $1\leqslant i\leqslant k$. Therefore, the characteristic polynomials of $\rho_1(g)$ and $\rho_2(g)$ are the same for all $g\in G$, the corollary of the  Brauer-Nesbitt theorem applies.
\end{proof}

The assumption $p>n$ in Proposition~2.3 cannot be dropped: if $\chi$ is a
non-trivial character of $G$ with values in $E^\times$, then the direct sums of trivial representation
$\mathbf 1^{\oplus p}$ and $\chi^{\oplus p}$ are semisimple, non-isomorphic,
and both have identically vanishing trace. In \S3 we replace the condition
$p>n$ by absolute irreducibility of one of the two representations, the key
input is the following Burnside's well-known theorem:

\begin{lem}[Burnside]\label{lem:burnside}
Let $G$ be a group, $E$ a field, $V$ a finite-dimensional $E$-vector space
and $\rho\colon G\to\GL_E(V)$ an absolutely irreducible representation.
Then the $E$-subalgebra $R:=E[\rho(G)]$ of $\End_E(V)$ equals $\End_E(V)$.
\end{lem}
\begin{proof}
Let $\Ebar$ be an algebraic closure of $E$ and $\overline V=V\otimes_E\Ebar$.
Absolute irreducibility means that $\overline V$ is an irreducible
representation of $G$, so Burnside's theorem over the algebraically closed
field $\Ebar$ \cite[Ch.~XVII, Corollary~3.4]{Lan02} 
gives $\Ebar[\rho(G)]=\End_{\Ebar}(\overline V)$. The inclusion $R\subset\End_E(V)$ induces an injection
\[
R\otimes_E\Ebar\hookrightarrow\End_E(V)\otimes_E\Ebar=\End_{\Ebar}(\overline V),
\]
whose image is $\Ebar[\rho(G)]$. Hence $\dim_E R=\dim_{\Ebar}\Ebar[\rho(G)]=(\dim_E V)^2=\dim_E\End_E(V)$,
and since $R\subset\End_E(V)$, we get $R=\End_E(V)$.
\end{proof}

\begin{defi}\label{delta}
Let $K$ be a global function field, $\Sigma_K$ its set of places, and for
$v\in\Sigma_K$ let $N(v)$ be the cardinality of its residue field. For
$S\subset\Sigma_K$ the \emph{Dirichlet density} and the \emph{upper
Dirichlet density} of $S$ are
\[
\delta(S)=\lim_{s\to1^+}\frac{\sum_{v\in S}N(v)^{-s}}{\sum_{v\in\Sigma_K}N(v)^{-s}},
\qquad
\udens(S)=\limsup_{s\to1^+}\frac{\sum_{v\in S}N(v)^{-s}}{\sum_{v\in\Sigma_K}N(v)^{-s}},
\]
the first whenever the limit exists.

\end{defi}

\begin{thm}[Chebotarev density theorem {\cite[Thm.~9.13A]{Ros02}}]\label{thm:cheb}
Let $L/K$ be a finite Galois extension of global function fields with
group $G$, and let $C\subset G$ be a conjugacy class. The set of places
$v$ of $K$ unramified in $L$ with $\Frob_v\in C$ has Dirichlet density
$|C|/|G|$.
\end{thm}
 
Consequently, if $S$ is a finite set of places and $K_S\subset\Ksep$ is
the maximal extension unramified outside $S$, the Frobenius elements
$\Frob_v$, $v\notin S$, are dense in $\Gal(K_S/K)$. Hence a continuous class
function on $\Gal(K_S/K)$ that vanishes at $\Frob_v$ for almost all $v$
vanishes identically. In particular, if $\rho_1,\rho_2\colon\GK\to\GL_n(E)$
are continuous, unramified outside $S$, and
$\tr\rho_1(\Frob_v)=\tr\rho_2(\Frob_v)$ for almost all $v$, then
$\tr\rho_1=\tr\rho_2$ on $\GK$.

In 1998, Rajan studied strong multiplicity one for $\ell$-adic representations, and provided a finer result on representations having a positive density of places with equal Frobenius trace.  

\begin{thm}[\cite{Ra98}, Theorem 1] \label{Ra1}
Let $K$ be a global field, $G_K:=\Gal(K^{\rm sep}/K)$ be the absolute Galois group, and $\Sigma_K$ be the set of non-archimedean places of $K$. Let $F$ be a non-archimedean local field of characteristic $0$ and residue characteristic $\fl$. Consider representations $\rho: G_K\rightarrow \GL_r(F)$ such that $\rho$ is continuous, unramified outside a finite set $S\subset\Sigma_K$, and semisimple.

 Suppose $\rho_1, \rho_2$ are such representations and the upper density of ${\rm SM}(\rho_1,\rho_2):=\{v\in\Sigma_K-S\mid {\rm tr}(\rho_1({\rm Frob}_v))={\rm tr}(\rho_2({\rm Frob}_v))\}$ in $\Sigma_K$ is strictly greater than $1-1/2r^2$, then $\rho_1\simeq\rho_2$.

\end{thm}

As an application to his main result, the following corollary on representations associated to elliptic curves gives a finer result on Theorem \ref{0case}. It shows what happens when two $\ell$-adic representations associated to non-CM elliptic curves have equal Frobenius traces for a positive density of places.

\begin{cor}[\cite{Ra98}, Corollary 2]\label{Ra2}
Suppose $E_1, E_2$ are two elliptic curves defined over a number field $K$. Let $S$ be the finite set consisting of the ramified places of $E_1$ and $E_2$. For a rational prime $\fl$, let $\rho_i$ be the $\fl$-adic Galois representation associated to $E_i$ for $i=1,2$. For $v\in\Sigma_K-S$ with $(\fl,v)=1$, let $a_v(E_i)$ denote the Frobenius trace ${\rm tr}(\rho_i({\rm Frob}_v))$ associated to $\rho_i$.  

Suppose that the set $\{v\in\Sigma_K-S\mid a_v(E_1)=a_v(E_2)\}$ in $\Sigma_K$ has positive upper density. Then there exists a quadratic Dirichlet character $\chi$ of $G_K$ such that $E_2$ is isogenous to $E_1\otimes \chi$.

\end{cor}

\subsection{Drinfeld modules and associated Galois representations}
Let $A=\Fq[T]$ be the polynomial ring over $\Fq$, where $q=p^e$ is a
prime power, let $F=\Fq(T)$ be the field of fractions of $A$, and let $K$
be a finite extension of $F$. We write $\Kbar$ for an algebraic closure
of $K$, $\Ksep\subset\Kbar$ for the separable closure and
$\GK=\Gal(\Ksep/K)$. Set $K\{\tau\}$ to be the twisted polynomial ring with usual addition rule, and the multiplication rule is defined to be $\tau\alpha=\alpha^q\tau \text{ for any } \alpha\in K$.  

We view $K$ as an $A$-field, which is a field equipped with a homomorphism $\gamma: A\rightarrow K$. The {\emph{$A$-characteristic}} of $K$ is defined to be the kernel of $\gamma$. A \emph{Drinfeld $A$-module} of rank $r\ge1$ over $K$ is an
$\mathbb F_q$-algebra homomorphism $\phi\colon A\to K\{\tau\}$,
$a\mapsto\phi_a$, such that
\[
\phi_T=\gamma(T)+g_1\tau+\cdots+g_r\tau^r,\qquad g_1,\dots,g_r\in K,\ g_r\neq0 .
\]
Since $A=\mathbb F_q[T]$, $\phi$ is determined by $\phi_T$. We say that $\phi$
has \emph{generic characteristic} if $\gamma$ is injective, i.e. its $A$-characteristic is $0$. For Drinfeld
modules $\phi_1,\phi_2$ over $K$, a \emph{morphism} $u\colon\phi_1\to\phi_2$
over $K$ is an element $u\in K\{\tau\}$ with $u\phi_{1,a}=\phi_{2,a}u$ for all
$a\in A$. We write $\mathrm{Hom}_K(\phi_1,\phi_2)$ for the $A$-module of such
morphisms defined over $K$, and $\mathrm{End}_K(\phi)=\mathrm{Hom}_K(\phi,\phi)$. A non-zero
morphism between Drinfeld modules of the same rank is an \emph{isogeny}. We
say $\phi$ is \emph{non-CM} if $\mathrm{End}_{\overline K}(\phi)=A$.
\begin{defi}
\begin{enumerate}
\item[(i)] Let $\phi$ be a Drinfeld $A$-module over $K$. The Drinfeld module $\phi$ gives $\overline{K}$ an $A$-module structure, where $a\in A$ acts on $\alpha\in \overline{K}$ via $a\cdot\alpha=\phi_{a}(\alpha)$. We use the notation $^{\phi}\overline{K}$ to emphasize the action of $A$ on $\overline{K}$.

\item[(ii)] Let $\mathfrak{a}$ be an ideal of $A$ with monic generator $a$.The {{$\mathfrak{a}$-torsion}} of $\phi$ is defined to be the set $$\phi[\mathfrak{a}]=\phi[a]:=\{ \alpha\in \overline{K}\mid \phi_a(\alpha)=0 \}\subseteq \overline{K}.$$ The action $b\cdot \alpha=\phi_b(\alpha) \ \textrm{ for all } b\in A \textrm{ and } \alpha\in\phi[\mathfrak{a}]$ also gives $\phi[\mathfrak{a}]$ an $A$-module structure. 
\end{enumerate}
\end{defi}

\begin{prop}\label{prop0.2}
Let $\phi$ be a rank $r$ Drinfeld module over $K$ and $\mathfrak{a}$ a non-zero ideal of $A$.
If $\phi$ has $A$-characteristic prime to $\mathfrak{a}$, then the $A/\mathfrak{a}$-module $\phi[\mathfrak{a}]$ is free of rank $r$.
\end{prop}
\begin{proof}
See \cite{G96} Proposition 4.5.7.
\end{proof}

Let $\phi$ be a rank $r$ Drinfeld module over $K$ of generic characteristic, then $\phi_a(x)$ is separable, so we have $\phi[a]\subseteq K^{{\rm{sep}}}$.
This implies that $\phi[a]$ has a $G_K$-module structure. Given a nonzero prime ideal $\fl$ of $A$ with prime power $\fl^i$, we can consider the $G_K$-module $\phi[\mathfrak{l}^i]$. We obtain the residual representation 
$$\bar{\rho}_{\phi,\mathfrak{l}^i}:G_K\longrightarrow {\rm{Aut}}(\phi[\mathfrak{l}^i])\cong \GL_r(A/{\mathfrak{l}}^i).$$
In the case $i=1$, we get the so-called {mod $\mathfrak{l}$ Galois representation}
$$\bar{\rho}_{\phi,\mathfrak{l}}:G_K\longrightarrow {\rm{Aut}}(\phi[\mathfrak{l}])\cong \GL_r(A/\fl).$$
Taking the inverse limit over $i$, we obtain the $\fl$-adic Tate module
$T_\fl(\phi)=\varprojlim_i\phi[\fl^i]$, a free $\Al$-module of rank $r$,
and the $\fl$-adic Galois representation
\[
\rho_{\phi,\fl}\colon\GK\longrightarrow\Aut_{\Al}(T_\fl(\phi))\cong\GL_r(\Al)\subset\GL_r(\Fl),
\]
where $\Al$ and $\Fl$ are the completions of $A$ and $F$ at $\fl$. We put
$V_\fl(\phi)=T_\fl(\phi)\otimes_{\Al}\Fl$. Two facts are used
throughout: $V_\fl(\phi)$ is a semisimple $\GK$-module
\cite{Tag93}, and (Tate conjecture \cite{Tag95}) for Drinfeld
modules $\phi_1,\phi_2$ over $K$ of generic characteristic,
\[
\Hom_K(\phi_1,\phi_2)\otimes_A\Al\;\xrightarrow{\ \sim\ }\;\Hom_{\GK}\bigl(T_\fl(\phi_1),T_\fl(\phi_2)\bigr).
\]

\section{Frobenius Traces and Isomorphism of Galois Representations}
As an alternative approach to Proposition \ref{pcase}, we observe that if one restricts to the representations arising from rank-$2$ Drinfeld modules, one can still rely on Frobenius traces to distinguish Drinfeld modules up to isogeny classes. This method avoids the assumption ${\rm char}(E)>n$, while it requires a closed formula between coefficients of Drinfeld modules and coefficients of Frobenius characteristic polynomials, which in rank $r\geqslant 3$ is not clear.

\begin{lem}\label{existh}

Let $\phi_1,\phi_2$ be Drinfeld modules of rank $2$ over $K$ of generic
characteristic, $\phi_{1,T}=T+g_1\tau+g_2\tau^2$ and
$\phi_{2,T}=T+s_1\tau+s_2\tau^2$, such that for some prime $\fl$ of $A$,
$\tr\rho_{\phi_1,\fl}(\Frob_v)=\tr\rho_{\phi_2,\fl}(\Frob_v)$ for all but
finitely many places $v$. Put $L:=K\bigl(g_2^{1/(q-1)},s_2^{1/(q-1)}\bigr)\subset\Ksep$. Then
$L/K$ is a finite abelian extension of exponent dividing $q-1$, and
$\rho_{\phi_1,\fl}|_{G_L}\cong\rho_{\phi_2,\fl}|_{G_L}$ for every prime
$\fl$ of $A$.

\end{lem}
\begin{proof}
Since $p\nmid q-1$ and $\Fq^\times\subset K$, Kummer theory shows that $L/K$ is separable and abelian of exponent dividing $q-1$. One may choose $L$ to be the compositum of Kummer extensions $L=K(\sqrt[q-1]{g_2},\sqrt[q-1]{s_2})$.
Since ${\rm tr}(\rho_{1, \fl}({\rm Frob}_v))={\rm tr}(\rho_{2, \fl}({\rm Frob}_v))\textrm{ for all but finitely many places $v$},$ the Chebotarev density theorem (Theorem~\ref{thm:cheb}) implies ${\rm tr}(\rho_{1,\fl}(g))={\rm tr}(\rho_{2,\fl}(g))$ for any $g\in G_K$.
 On the other hand, from the Weil pairing \cite{vdH04} for Drinfeld modules, we have that $$\det\circ\rho_{i,\fl}=\rho_{\psi_i,\fl} \textrm{ for $i=1,2$},$$ where $\psi_{1,T}=T-g_2\tau$, and $\psi_{2,T}=T-s_2\tau$ are twisted Carlitz modules defined over $K$ together with their associated $\fl$-adic Galois representations $\rho_{\psi_i,\fl}:G_K\rightarrow \GL_1(F_\fl)$. The field $L$ is chosen so that both $\psi_1$ and $\psi_2$ become isomorphic to the same twisted Carlitz module $\psi_T=T-\tau$ over $L$.
 Hence we have $$\det\rho_{\phi_1,\fl}|_{G_L}=\det\rho_{\phi_2,\fl}|_{G_L}.$$
Therefore, $\rho_{\phi_1,\fl}(h)$ and $\rho_{\phi_2,\fl}(h)$ have the same characteristic polynomial for every $h\in G_L$. Both restrictions are semisimple by \cite[Theorem 0.1]{Tag93}, so Theorem \ref{BN} implies $\rho_1\cong\rho_2$ when restricting to $G_L$.

\end{proof}

\begin{lem}\label{1coho}
Let $G$ be a profinite group, $H\leqslant G$ an open normal subgroup, $E$ a
local field and $V$ an $n$-dimensional $E$-vector space. Let
$\rho_1,\rho_2: G\rightarrow\GL_E(V)$ be continuous representations such
that
\begin{enumerate}
\item $\rho_1(G)$ is open in $\GL_E(V)$,
\item there is $\Phi\in\GL_E(V)$ with $\Phi\rho_1(h)\Phi^{-1}=\rho_2(h)$ for all $h\in H$.
\end{enumerate}
Then there is a continuous character $\chi\colon G\to E^\times$, trivial
on $H$, such that $\rho_1\cong\rho_2\otimes\chi$.

\end{lem}

\begin{proof}
For $g\in G$ put $\Psi(g):=\Phi\rho_1(g)\Phi^{-1}\rho_2(g)^{-1}$. For
$h\in H$, normality of $H$ gives
\[
\rho_2(h)\Psi(g)=\Phi\rho_1(g)\rho_1(g^{-1}hg)\Phi^{-1}\rho_2(g)^{-1}
=\Phi\rho_1(g)\Phi^{-1}\rho_2(g^{-1}hg)\rho_2(g)^{-1}=\Psi(g)\rho_2(h),
\]
so $\Psi(g)$ lies in $\Aut_H(V)$, the group of automorphisms of $V$
commuting with $\rho_2(H)=\Phi\rho_1(H)\Phi^{-1}$.
 
We claim $\Aut_H(V)=E^\times$. Since $H$ is open in the compact group $G$, it has finite index, then
$\rho_1(G)$ is a finite union of cosets of $\rho_1(H)$, each of which is compact and hence closed in
$\rho_1(G)$. The complement of $\rho_1(H)$ in $\rho_1(G)$ is a finite union of some
of these cosets, hence closed, so $\rho_1(H)$ is open in $\rho_1(G)$, and
therefore open in $\mathrm{GL}_E(V)$ by hypothesis~(1). Let $\mathcal O_E$ be the valuation
ring of $E$ and $\pi$ a uniformiser. By choosing an $\mathcal O_E$-lattice
$\Lambda\subset V$, the subgroups $1+\pi^i\mathrm{End}_{\mathcal O_E}(\Lambda)$, $i\ge1$, form a
basis of open neighbourhoods of $1$ in $\mathrm{GL}_E(V)$. Thus, the open subgroup
$\rho_1(H)$ contains some $1+\pi^k\mathrm{End}_{\mathcal O_E}(\Lambda)$, whose $E$-linear span is $\End_E(V)$.  This forces the commutant of $\rho_1(H)$ to be the centre $E$ of
$\End_E(V)$, and so is the commutant of $\rho_2(H)$. ${\rm Aut}_H(V)$ is its units $E^\times$.
 
Hence $\Psi$ takes values in $E^\times$, and since scalars are central,
\[
\Psi(g_1g_2)=\Phi\rho_1(g_1)\Phi^{-1}\,\Psi(g_2)\,\rho_2(g_1)^{-1}=\Psi(g_1)\Psi(g_2).
\]
So $\chi:=\Psi$ is a continuous character, it is trivial on $H$ by hypothesis~(2),
and $\Phi\rho_1(g)\Phi^{-1}=\chi(g)\rho_2(g)$ for all $g\in G$, we have\
$\rho_1\cong\rho_2\otimes\chi$.

\end{proof}

\begin{thm}\label{mainthm1}
Let $q=p^e$ be a prime power, $A:=\F_q[T]$ with field of fractions $F:=\F_q(T)$. Let $\phi_1$, $\phi_2$ be two non-CM (see \S2.2 for definition) rank-$2$ Drinfeld modules over $K$, a finite extension of $F$, of generic characteristic. Let $\fl$ be a prime ideal of $A$ and $\rho_{i,\fl}:G_K\rightarrow \GL_2(F_\fl)$  be the $\fl$-adic Galois representations associated to $\phi_i$ for $i=1,2$. 

Suppose that the Frobenius traces ${\rm{tr}}(\rho_{i,\fl}({\rm Frob}_{v}))$ are equal for all but finitely many places $v$ of $K$, then $\phi_1$ and $\phi_2$ are $K$-isogenous Drinfeld modules. Consequently, $\rho_{1,\fl}\cong\rho_{2,\fl}$.
\end{thm}

\begin{proof}

Fix a prime $\fl$ of $A$ and write $\rho_i=\rho_{\phi_i,\fl}$. Since
$\End_{\Kbar}(\phi_i)=A$, the open image theorem \cite[Thm.~0.1]{PR09} shows that
$\rho_i(\GK)$ is open in $\GL_2(\Fl)$. By Lemma~\ref{existh},
$\rho_1|_{G_L}\cong\rho_2|_{G_L}$ for a finite Galois extension $L/K$.
Lemma~\ref{1coho}, applied with $G=\GK$ and $H=G_L$, gives a continuous character
$\chi\colon\GK\to\Fl^\times$, trivial on $G_L$, with
$\rho_1\cong\rho_2\otimes\chi$.
 
Suppose $\chi\neq1$. Taking traces, we have $\tr\rho_2(\Frob_v)\bigl(1-\chi(\Frob_v)\bigr)=0$
for almost all $v$.  By Theorem~\ref{thm:cheb}, the continuous class function
$g\mapsto\tr\rho_2(g)(1-\chi(g))$ vanishes on $\GK$. Choose $g_0$ with
$\chi(g_0)\ne1$ and set $N=\ker\chi$, an open normal subgroup of $G_K$. The trace $\tr\rho_2$
vanishes on $g_0N$. Note that $\rho_2(g_0N)$ is a non-empty open
subset of $\GL_2(\Fl)$ because $\rho_2(N)$ is open in $\rho_2(\GK)$, hence $\rho_2(g_0N)$ is open in $M_2(F_\fl)$. The trace function ${\rm tr}$ as a linear function on $M_2(F_\fl)$ vanishes on a non-empty open subset $\rho_2(g_0N)$ implies ${\rm tr}=0$, a contradiction. Hence $\chi=1$ and $\rho_1\cong\rho_2$. Consequently $V_\fl(\phi_1)\cong V_\fl(\phi_2)$ as $\GK$-modules, so we have
$$0\neq\Hom_{\GK}(T_\fl(\phi_1),T_\fl(\phi_2))\cong\Hom_K(\phi_1,\phi_2).$$ 
A non-zero homomorphism between Drinfeld modules of the same rank is an isogeny.

\end{proof}

The key strategy in Theorem \ref{mainthm1} is that one can take a suitable finite extension $L/K$ to make the Frobenius determinants equal, then the fact that the Frobenius characteristic polynomial depends only on the trace and determinant forces the Frobenius characteristic polynomial of two representations to coincide when restricted to $G_L$. Hence this technique would not work for higher dimensional Galois representations even for representations associated to rank-$r$ Drinfeld modules with $r\geqslant 3$. This leads us to a general algebraic approach, which applies  to arbitrary $n$-dimensional semisimple $\fl$-adic Galois representations over a local field $E$ of positive characteristic, under suitable irreducibility conditions.

\begin{prop}\label{Main1}
Let $K$ be a global function field and $E$ be any local field of characteristic $p>0$. Let $n\in \mathbb{N}$, and $\rho_1, \rho_2:G_K\rightarrow \GL_n(E)$ be continuous representations, unramified outside a finite set $S$ of places of $K$. Suppose further that
\begin{itemize}
\item For any place $ v\not\in S$, we have ${\rm Tr}(\rho_1({\rm Frob}_v))={\rm Tr}(\rho_2({\rm Frob}_v))$.
\item $\rho_1$ is absolutely irreducible and $\rho_2$ is irreducible.
\end{itemize}
Then $\rho_1\cong\rho_2$.
\end{prop}

\begin{proof}

We start with the fact that for $i=1,2$,  the composition $${\rm Tr}\circ \rho_i: G_K\rightarrow E\ \ \ g\mapsto {\rm Tr}(\rho_i(g))$$ is
continuous and invariant under conjugation. Hence Theorem~\ref{thm:cheb} implies that ${\rm Tr}(\rho_1(g))={\rm Tr}(\rho_2(g))$ for any $g\in G_K$.

By Lemma~\ref{lem:burnside}, we may write $R_1=E[\rho_1(G_K)]={\rm End}_E(V_1)$ and $R_2=E[\rho_2(G_K)]$, where $V_1$ is a $n$-dimensional vector space over $E$ with $G_K$-action defined via $\rho_1$, similar for the vector space $V_2$. We fix a $E$-basis $\{a_i\}_{i=1}^{n^2}$ of $R_1$, where $a_i=\rho_1(g_i)$ for some $g_i\in G_K$. Let $\{a_i^*\}_{i=1}^{n^2}$ be the trace-dual basis of ${\rm End}_E(V_1)$, i.e. ${\rm Tr}(a_i^*a_j)=\delta_{ij}$. Then we define the $E$-linear map

$$
\begin{array}{lrl}
\Pi:&{\rm Hom}_E(V_1,V_2)&\rightarrow {\rm Hom}_E(V_1,V_2)\\
&S& \mapsto \sum_{i=1}^{n^2}\rho_2(g_i)Sa_i^*
\end{array}
$$
On the other hand, we define the following $E$-algebra homomorphism

$$\Phi:{\rm End}_E(V_1)\rightarrow E[\rho_2(G_K)]\ \ \sum_{i=1}^{n^2} c_ia_i\mapsto \sum_{i=1}^{n^2}c_i\rho_2(g_i).$$

Since $\{a_i\}$ is a basis of $R_1$, the map $\Phi$ is a
well-defined $E$-linear map. We check that $\Phi(\rho_1(g))=\rho_2(g)$ for
every $g\in G_K$. Let
\[
I=\{x\in R_2:\ \mathrm{Tr}(xy)=0\ \text{for all }y\in R_2\},
\]
be an $E$-subspace of $R_2$. It is a two-sided ideal because for $x\in I$ and
$a,y\in R_2$, we have
\[
\mathrm{Tr}\bigl((ax)y\bigr)=\mathrm{Tr}\bigl(x(ya)\bigr)=0,
\qquad
\mathrm{Tr}\bigl((xa)y\bigr)=\mathrm{Tr}\bigl(x(ay)\bigr)=0.
\]
As $\rho_2$ is irreducible, $V_2$ is a
faithful simple $R_2$-module, so $R_2$ is a simple algebra by Wedderburn's
theorem \cite[Ch.XVII Corollary~3.5]{Lan02}. Hence $I=0$ or $I=R_2$. Suppose $I=R_2$. Taking $y=1$ gives
$\mathrm{Tr}(\rho_2(g))=0$ for all $g\in G_K$, hence
$\mathrm{Tr}(\rho_1(g))=0$ for all $g$ by the trace equality. Then for all
$g,h\in G_K$,
\[
\mathrm{Tr}\bigl(\rho_1(g)\rho_1(h)\bigr)=\mathrm{Tr}\bigl(\rho_1(gh)\bigr)=0,
\]
and since the $\rho_1(g)$ span $R_1=\mathrm{End}_E(V_1)$ by
Lemma~\ref{lem:burnside}, the trace form of $\mathrm{End}_E(V_1)$ would
vanish identically, a contradiction. Hence $I=0$, the trace pairing on $R_2$ is non-degenerate.

Now we fix a $g\in G_K$ and
write $\rho_1(g)=\sum_i c_ia_i=\sum_i c_i\rho_1(g_i)$, so
$\Phi(\rho_1(g))=\sum_i c_i\rho_2(g_i)$. Put
$z:=\sum_i c_i\rho_2(g_i)-\rho_2(g)\in R_2$. For every $h\in G_K$, since
$\rho_2$ is a homomorphism,
\[
\mathrm{Tr}\bigl(z\,\rho_2(h)\bigr)
=\sum_i c_i\,\mathrm{Tr}\bigl(\rho_2(g_ih)\bigr)-\mathrm{Tr}\bigl(\rho_2(gh)\bigr).
\]
By the trace equality applied to the elements $g_ih$ and $gh$ of $G_K$,
this equals
\[
\sum_i c_i\,\mathrm{Tr}\bigl(\rho_1(g_ih)\bigr)-\mathrm{Tr}\bigl(\rho_1(gh)\bigr)
=\mathrm{Tr}\Bigl(\bigl(\textstyle\sum_i c_i\rho_1(g_i)-\rho_1(g)\bigr)\rho_1(h)\Bigr)=0 .
\]
Because $\rho_2(h)$ span $R_2$, we get $\mathrm{Tr}(zy)=0$ for all
$y\in R_2$, so $z\in I=\{0\}$. Thus $$\Phi(\rho_1(g))=\rho_2(g).$$ Consequently, $\Phi$ is indeed an $E$-algebra homomorphism. It suffices to note that
$\Phi(a_ia_j)=\Phi(\rho_1(g_ig_j))=\rho_2(g_ig_j)=\rho_2(g_i)\rho_2(g_j)
=\Phi(a_i)\Phi(a_j)$, and $\Phi(1)=\rho_2(1)=1$.

Now we claim that for $S\in {\rm Hom}_E(V_1,V_2)$, the map $\Pi(S):V_1\rightarrow V_2$ is $G_K$-equivariant, and $\Pi$ is not a zero map. Given a fixed $g\in G_K$, we can write 
\begin{equation}
\rho_1(g) a_i = \sum_{j=1}^{n^2} \text{Tr}(a_j^* \rho_1(g) a_i) a_j \tag{1}, \textrm{ and set }c_{ij}(g)=\text{Tr}(a_j^* \rho_1(g) a_i).\end{equation}

Similarly, we have \begin{equation}
a_j^* \rho_1(g) = \sum_{i=1}^{n^2} \text{Tr}(a_j^* \rho_1(g) a_i) a_i^* \tag{2}, \textrm{ and set }d_{ji}(g)=\text{Tr}(a_j^* \rho_1(g) a_i).\end{equation}
One can easily see that $c_{ij}(g)=d_{ji}(g)$. Now we compute
$$
\begin{array}{ll}
\rho_2(g)\Pi(S)&=\sum_{i=1}^{n^2}\rho_2(gg_i)Sa_i^*=\sum_{i=1}^{n^2}\Phi(\rho_1(g)a_i)Sa_i^*=\sum_{i=1}^{n^2}\sum_{j=1}^{n^2}c_{ij}(g)\Phi(a_j)Sa_i^*\\
&\ \\
&=\sum_{i=1}^{n^2}\sum_{j=1}^{n^2}c_{ij}(g)\rho_2(g_j)Sa_i^*=\sum_{j=1}^{n^2}\rho_2(g_j)S\left(\sum_{i=1}^{n^2} d_{ji}(g)a_i^*\right)\\
&\ \\
&=\sum_{j=1}^{n^2}\rho_2(g_j)Sa_j^*\rho_1(g)=\Pi(S)\rho_1(g).
\end{array}
$$
It remains to prove $\Pi$ is not a zero map, we start with the operator $\Pi(S)=\sum_i \rho_2(g_i)\, S\, a_i^\ast.$
Consider the equality $\Phi(a_i)=\rho_2(g_i)$. We have seen that $\Phi:{\rm End}_E(V_1)\rightarrow E[\rho_2(G_K)]\subset {\rm End}_E(V_2)$ is an $E$-algebra homomorphism. As ${\rm End}_E(V_1)$ is a simple algebra, $\Phi$ must be injective. Comparing $E$-dimensions on ${\rm End}_E(V_i)$ for $i=1,2$ forces $\Phi$ to be an $E$-algebra isomorphism. Thus by the Skolem--Noether theorem we may write
\[
\Phi(a_i)=Ta_iT^{-1}
\]
for some $E$-linear isomorphism $T:V_1 \to V_2$. Hence we obtain
\[
\Pi(S)=\sum_i T a_i T^{-1} S a_i^\ast.
\]
Now choose $S = TX$ for some $X \in {\rm End}_E(V_1)$. Then
\[
\Pi(TX)
=
\sum_i T a_i T^{-1} T X a_i^\ast
=
T\left(\sum_i a_i X a_i^\ast\right).
\]
Using the identity
$$\sum_i a_i X a_i^\ast = {\rm Tr}(X)\,\mathrm{Id},$$
we get $\Pi(TX)={\rm Tr}(X)\,T.$ Put $X=E_{11}$ to be the elementary matrix with $(1,1)$-entry equal to $1$, we have ${\rm Tr}(E_{11})=1$, hence
\[
\Pi(T E_{11}) = T \neq 0.
\]
This proves the claim.

Finally, we choose a  non-zero map $S$ such that $\Pi(S)\neq0$, then apply Schur's lemma on the nonzero $G_K$-module homomorphism $\Pi(S):V_1\rightarrow V_2$ to get an isomorphism between $V_1$ and $V_2$, hence $\rho_1\cong\rho_2$.

\end{proof}

\begin{rem}\label{irred}
The assumption that $\rho_1$ in Proposition~\ref{Main1} be absolutely irreducible
is used in the proof, and it cannot be weakened to the irreducibility of
both representations. Let $E'/E$ be a purely inseparable extension of
degree $p$ and $\chi\colon\GK\to E'^\times$ a continuous character with
$E[\chi(\GK)]=E'$. Regarding $E'$ as an $E$-vector space, multiplication
by $\chi$ defines $\rho_\chi\colon\GK\to\GL_E(E')\cong\GL_p(E)$.
\begin{enumerate}
\item $\rho_\chi$ is irreducible. This is because an $E$-subspace stable under $\rho_\chi(\GK)$
is stable under $E[\chi(\GK)]=E'$, hence is $\{0\}$ or $E'$.
\item $\tr\rho_\chi(g)$ is equal to the field trace $\Tr_{E'/E}(\chi(g))=0$ for all $g$, since the trace
of a purely inseparable extension vanishes.
\item If $\rho_\chi\cong\rho_{\chi'}$ via the isomorphism $f$, then $x\mapsto fxf^{-1}$ is an
$E$-algebra isomorphism $E'=E[\rho_\chi(\GK)]\to E[\rho_{\chi'}(\GK)]=E'$
sending $\chi(g)$ to $\chi'(g)$. As $E'/E$ has no non-trivial  
$E$-automorphisms, we must have $\chi=\chi'$.
\end{enumerate}
Hence two distinct such characters give irreducible, non-isomorphic
representations with equal traces. Example~\ref{ex:CM} shows that such
pairs can arise from Drinfeld modules with CM.

\end{rem}

The corollary below is an application of Proposition \ref{Main1} on $\fl$-adic Galois representations arising from generic characteristic Drinfeld modules of arbitrary rank:

\begin{cor}\label{corr}
Let $q=p^e$ be a prime power, $A:=\F_q[T]$ with field of fractions $F:=\F_q(T)$. Let $r\in \mathbb{Z}_{>0}$, and $\phi_1$, $\phi_2$ be two non-CM rank-$r$ Drinfeld modules over $K$, a finite extension of $F$, of generic characteristic. Let $\fl$ be a prime ideal of $A$. Let $\rho_{i,\fl}:G_K\rightarrow \GL_r(F_\fl)$  be the $\fl$-adic Galois representations associated to $\phi_i$ for $i=1,2$. 

Suppose that the Frobenius traces ${\rm{tr}}(\rho_{i,\fl}({\rm Frob}_{v}))$ are equal for all but finitely many places $v$ of $K$, then
$\phi_1$ and $\phi_2$ are $K$-isogenous and so $\rho_{1,\fl}\cong\rho_{2,\fl}$.

\end{cor}

\begin{proof}
 
From the fact that Frobenius traces are independent of the choice of prime $\fl$, and the  open image theorem for non-CM Drinfeld modules in \cite[Theorem 0.1]{PR09}, one can  deduce that there is a prime $\tilde{\fl}$ of sufficiently large degree so that $\rho_{1,\tilde{\fl}}$ is absolutely irreducible and $\rho_{2,\tilde{\fl}}$ is irreducible. Thus we have $\rho_{1,\tilde{\fl}}\cong\rho_{2,\tilde{\fl}}$ by Proposition \ref{Main1}. Then the same argument using the Tate isogeny theorem as in the proof of Theorem \ref{mainthm1} follows.

\end{proof}

\begin{ex}[CM Drinfeld modules with equal traces]\label{ex:CM}
Assume $q>2$. Let $\theta\in\Kbar$ with $\theta^p=T$, and put
$K=F(\theta)=\Fq(\theta)$ and $B=\Fq[\theta]\supset A$. Let $\psi$ be the
Carlitz module of $B$, $\psi_\theta=\theta+\tau$, and $\psi'$ be its twist
$\psi'_\theta=\theta+\theta\tau$. Restricting to $A$ gives
\[
\phi_T:=\psi_{\theta^p}=(\theta+\tau)^p,\qquad
\phi'_T:=\psi'_{\theta^p}=(\theta+\theta\tau)^p\ \in K\{\tau\}.
\]
Then:
\begin{enumerate}
\item[(a)] $\phi,\phi'$ are Drinfeld $A$-modules of rank $p$ over $K$ of
generic characteristic with $B$ contained in both $\End_K(\phi)$ and $\End_K(\phi')$. Their CM-field
$F(\theta)$ is purely inseparable over $F$.
\item[(b)] $\tr\rho_{\phi,\fl}=\tr\rho_{\phi',\fl}=0$ identically, for every $\fl$.
\item [(c)]$\phi$ and $\phi'$ are isomorphic over $K(\theta^{1/(q-1)})$ but not
isogenous over $K$, so $\rho_{\phi,\fl}\not\cong\rho_{\phi',\fl}$.
\item[(d)] $\rho_{\phi,\fl}$ and $\rho_{\phi',\fl}$ are irreducible over $\Fl$.
\end{enumerate}
For $p=2$ and $q=2^e\ge4$ these modules have rank $2$, so the non-CM
hypothesis in Theorem~\ref{mainthm1} cannot be dropped.
\end{ex}

\begin{proof}
(a) Since $\psi$ and $\psi'$ are ring homomorphisms, $\phi_T=\psi_\theta^{\,p}$
and $\phi'_T=\psi_\theta'^{\,p}$ are $p$-powers in the twisted
polynomial ring $K\{\tau\}$, so $B$ lies in both $\End_K(\phi)$ and $\End_K(\phi')$. Besides, $\phi$ and $\phi'$ have constant term $\theta^p=T$
and $\tau$-degree $p$ since the leading coefficient of $\phi_T$ is $1$, and that
of $\phi'_T$ is $\theta\cdot\theta^{q}\cdots\theta^{q^{p-1}}=\theta^{(q^p-1)/(q-1)}\neq0$. Hence they both are
Drinfeld $A$-modules of rank $p$ of generic characteristic. Thus $F(\theta)$, as the field of fractions of $B$ with $[F(\theta):F]=p$, is their CM-field.

(b) Write $\fl=(l)$ with $l\in A$ monic. Since $\Fq$ is perfect, $l=\tilde l(\theta)^p$,
where $\tilde l$ has the $p$-th roots of the coefficients of $l$. Thus
$\fl B=\lambda^p$ with $\lambda=(\tilde l)$, and
$F(\theta)_\lambda:=F(\theta)\otimes_F\Fl$ is a field, purely
inseparable of degree $p$ over $\Fl$. As $\phi[\fl^i]=\psi[\lambda^{pi}]$,
we have $T_\fl(\phi)=T_\lambda(\psi)$, free of rank one over $B_\lambda$,
and $\GK$ acts through a character $\chi_\psi\colon\GK\to B_\lambda^\times$.
Thus $\rho_{\phi,\fl}(g)$ is multiplication by $\chi_\psi(g)$ on the
$\Fl$-vector space $F(\theta)_\lambda$.  In other words,
$\rho_{\phi,\mathfrak l}\cong\mathrm{Res}_{F(\theta)_\lambda/F_{\mathfrak l}}(\chi_\psi)$
is the restriction of scalars of the one-dimensional representation
$\chi_\psi$. In particular,
$\mathrm{tr}\,\rho_{\phi,\mathfrak l}=\mathrm{Tr}_{F(\theta)_\lambda/F_{\mathfrak l}}\circ\chi_\psi=0$.
The same argument applies to $\psi'$.

(c) Let $u\in\Ksep$ with $u^{q-1}=\theta$. Then
$u^{-1}(\theta+\tau)u=\theta+u^{q-1}\tau$, so $\psi'=u^{-1}\psi u$ and
$\phi'=u^{-1}\phi u$, which means that $\phi$ and $\phi'$ are isomorphic over $K(\theta^{1/(q-1)})$. 

Now let $v$ be the place of $K$ given by a monic
irreducible $ \fP\in B$ of degree $d$, distinct from $\lambda$ and $\theta$. We compute the
characteristic polynomial of ${\rm Frob}_v$ at a prime $\fl$ such that $\fl$ coprime to $\fP$ in $B$.
For the Carlitz
module, $\psi_\fP\equiv\tau^d\pmod \fP$ \cite[Theorem~4.2.7]{P23}, so the
Frobenius endomorphism of the reduction acts on $T_\lambda(\psi)$ as multiplication by
$\fP$. Writing $\psi_\fP=\sum_jc_j\tau^j$ with $c_d=1$ and $\fP\mid c_j$
for $j<d$, we get
$\psi'_\fP=u^{-1}\psi_\fP u=\sum_jc_j\theta^{(q^j-1)/(q-1)}\tau^j
\equiv\zeta_\fP\tau^d\pmod \fP$, where
$\zeta_\fP\equiv\theta^{(q^d-1)/(q-1)}\pmod \fP$ is the $(q-1)$-th power residue
symbol $(\theta/\fP)_{q-1}\in\Fq^\times$. So Frobenius acts as multiplication by
$\zeta_P^{-1}\fP$ for $T_\lambda(\psi')$. Since $\fP\notin F$ (an element
of $B\cap F=\Fq[T]$ is a $p$-th power in $B$), its minimal polynomial
over $F$ is $X^p-\fP^p$, with $\fP^p\in A$. The characteristic polynomials of ${\rm Frob}_v$ are therefore
\[
X^p-\fP^p\quad(\text{for }\phi),\qquad X^p-\zeta_\fP^{-p}\fP^p\quad(\text{for }\phi').
\]
Note that $K(u)/K$ is a Kummer extension, and $\fP$ is unramified in $K(u)$. Thus $\fP$ splits
completely in $K(u)$ if and only if $X^{q-1}-\theta$ has a root modulo $\fP$, i.e.\ if and only if
$\zeta_\fP=1$. By Theorem~\ref{thm:cheb}, the places $\fP$ with
$\zeta_\fP\neq1$ therefore have density $1-1/(q-1)>0$. As Frobenius characteristic polynomials are isogeny invariants,  $\phi$ and $\phi'$ are not $K$-isogenous.

(d) By Remark~\ref{irred}(1) it is enough to check that $F_{\mathfrak l}[\chi(G_K)]=F(\theta)_\lambda$
for $\chi=\chi_\psi,\chi_{\psi'}$. As $[F(\theta)_\lambda:F_{\mathfrak l}]=p$ is
prime, it suffices to find one $g$ with $\chi(g)\notin F_{\mathfrak l}$. By (c),
$\chi_\psi(\mathrm{Frob}_v)=\fP$ and
$\chi_{\psi'}(\mathrm{Frob}_v)=\zeta_\fP^{-1}\fP$ for a place $v$ given by
$\fP\nmid\theta\,l$. Since $F(\theta)\otimes_FF_{\mathfrak l}$ is a field,
$F(\theta)\cap F_{\mathfrak l}=F$, and $\fP\notin F$, we know neither value
lies in $F_{\mathfrak l}$.

\end{proof}

\section{Function Field Analogue of Strong Multiplicity One}

\subsection{Notation}
Throughout this section, $K$ is a global function field, and $\Sigma_K$, $\delta$ and
$\overline{\delta}$ are as in Definition~\ref{delta}.

\begin{itemize}
\item For a linear algebraic group $H$ over $E$ with identity component
$H^\circ$, put $\Phi:=H/H^\circ$ and denote by $H^\varphi$ the connected
component corresponding to $\varphi\in\Phi$. For a Zariski-closed subset
$X$, put $\Psi:=\{\varphi\in\Phi \mid H^\varphi\subset X\}$.
\item $\mathcal O_E$ denotes the valuation ring of $E$, $\mathfrak m$ its
maximal ideal, and $\mu_G$ the normalised Haar measure on a compact group $G$.

\end{itemize}

\subsection{Algebraic Chebotarev density}
We generalize the Chebotarev density theorem (Theorem \ref{thm:cheb}) based on the so-called local thinness property.

\begin{thm}\label{thma}

Let $E$ be a non-archimedean local field of positive characteristic $p$, and $\rho: G_K\rightarrow M(E)$ be a continuous representation that is unramified outside of a finite set of places in $\Sigma_K$, where $M$ is an affine algebraic group over $E$.  Let $H=\overline{\rho(G_K)}^{\rm Zar}\subset M$ be the Zariski closure of $\rho(G_K)$ in $M$, with identity component $H^\circ$, and finite quotient $\Phi:=H/H^\circ$.  Let $X\subset M$ be a conjugation-stable Zariski closed subset defined over $E$. Define
$$\Sigma_X=\left\{v\in\Sigma_K\mid \rho \textrm{ unramified at } v,\ \rho({\rm Frob}_v)\in X(E) \right\} .$$
Assume the following componentwise thinness property:
\begin{equation}\tag{CT}\label{CT}
\textrm{For all $\varphi\not\in\Psi$}, \textrm{ the set }X(E)\cap\rho(G_K)\cap H^\varphi(E) \textrm{ has Haar measure $0$ in }\rho(G_K).
\end{equation}
Then $\Sigma_X$ has density $\delta(\Sigma_X)=\frac{|\Psi|}{|\Phi|}$.

\end{thm}

\begin{proof}

Denote $G=\rho(G_K)$ and $G^\circ=G\cap H^\circ(E)$, then we have the finite embedding $G/G^\circ\hookrightarrow \Phi$. Let $L/K$ be the finite Galois extension cut out by $G_K\rightarrow G/G^\circ$, i.e. $\Gal(L/K)\cong G/G^\circ \stackrel{\rho}\hookrightarrow\Phi$ . Let $v\in \Sigma_K$ be a place where $\rho$ is unramified at $v$. Then $\rho({\rm Frob}_v)$ lies in a connected component of $H$ containing $\rho({\rm Frob}_v)$. Define $$\Sigma_{\rm comp}=\left\{v\in\Sigma_K\ {\rm unramified }\mid \rho({\rm Frob}_v)\in\cup_{\phi\in\Psi}H^\phi(E)\right\}.$$
The Chebotarev density theorem (Theorem~\ref{thm:cheb}) then implies $\delta(\Sigma_{\rm comp})=\frac{|\Psi|}{|\Phi|}.$

Now we write $X\cap H=\left(\cup_{\phi\in\Psi}H^\phi \right)\cup Y$, where $Y=\bigcup_{\varphi\notin\Psi}(X\cap H^\varphi)\subset H$ is Zariski-closed satisfying the property that $\textrm{for any }\varphi\not\in\Psi,\ Y\cap H^\varphi=X\cap H^\varphi$ is a proper Zariski-closed subset of $H^\phi$. Then we obtain a decomposition $$\Sigma_X=\Sigma_{\rm comp}\sqcup \Sigma_{\rm thin},$$ where $\Sigma_{\rm thin}:=\left\{v\in \Sigma_K\ {\rm unramified}\mid \rho({\rm Frob}_v)\in Y(E)\right\}$. The set $Z:=\rho(G_K)\cap Y(E)$ is closed in
$G$, since $G$ is compact and $Y(E)$ is Zariski closed. It is
also stable under conjugation by $G$. Indeed, let $h\in H$. Since $X$ is stable under conjugation, conjugation by any $h\in H$ permutes components in $X$, so it preserves
$\Psi$ and permutes the components not in $\Psi$, hence $hYh^{-1}=Y$. In particular $gZg^{-1}=Z$ for all $g\in G\subset H(E)$. By (\ref{CT}), $\mu_G(Z)=0$, and we may write $$\Sigma_{\mathrm{thin}}=\{v\text{ unramified}\mid\rho(\Frob_v)\in Z\}.$$
By the claim below, we get $\delta(\Sigma_{\mathrm{thin}})=0$, and so $\delta(\Sigma_X)=\delta(\Sigma_{\mathrm{comp}})=|\Psi|/|\Phi|$.
\begin{claim*}
Let $Z\subset G$ be a closed, conjugation-stable subset with $\mu_G(Z)=0$.
Then  $\Sigma_Z:=\{v\text{ unramified}\mid\rho(\Frob_v)\in Z\}$ has density $0$.
\end{claim*}
\begin{proof}[Proof of the claim]
Fix a closed embedding $M\hookrightarrow\mathrm{GL}_N$. Since $G$ is a compact subgroup of $\mathrm{GL}_N(E)$, it
stabilises an $\mathcal O_E$-lattice in $E^N$. Replacing $\rho$ by $P\rho P^{-1}$ for a suitable $P\in\mathrm{GL}_N(E)$, we
may assume $G\subset\mathrm{GL}_N(\mathcal O_E)$. 

For $n\ge1$ let
$U_n=G\cap(1+\mathfrak m^n\M_N(\mathcal O_E))$ be an open normal subgroup of $G$, and set $G_n=G/U_n$. Consider
for each $n\ge1$ the natural projection $\pi_n\colon G\to G_n$, and $Z_n:=\pi_n(Z)$. Since the $\{U_n\}_{n\ge1}$ form a basis of neighbourhoods of $1$, the decreasing sets
$\pi_n^{-1}(Z_n)=ZU_n$ have intersection equal to closure of
$Z$, which is $Z$ itself. Since $\mu_G(\pi_n^{-1}(Z_n))=|Z_n|/|G_n|$, continuity
of measures gives  
\[
\lim_{n\to\infty}\frac{|Z_n|}{|G_n|}=\lim_{n\to\infty}\mu_G(ZU_n)=\mu_G(Z)=0 .
\]
Moreover
$\Sigma_Z\subset\{v\mid\pi_n\rho(\Frob_v)\in Z_n\}$, and $Z_n$ is a union of
conjugacy classes of $G_n$, so Theorem~\ref{thm:cheb} for the extension
cut out by $\pi_n\circ\rho$ gives $\udens(\Sigma_Z)\le|Z_n|/|G_n|$ for every
$n$. Hence $\udens(\Sigma_Z)=0$, which implies $\delta(\Sigma_Z)=0$.
\end{proof}
\end{proof}

\subsection{Positive Characteristic SMO}
The main goal of this subsection is the following theorem:

\begin{thm}\label{thmb}
Let $K$ be a global function field, and  $\rho_1,\rho_2:G_K\rightarrow \GL_n(E)$ be two semisimple Galois representations over a local field $E$ of positive characteristic. Denote by
\begin{itemize}
\item $\rho:=(\rho_1,\rho_2):G_K\rightarrow \GL_n(E)\times\GL_n(E).$
\item$ H:=\overline{\rho(G_K)}^{\rm Zar}\subset \GL_n\times\GL_n$ is the Zariski closure of $\rho(G_K)$ in the affine algebraic group $\GL_n\times\GL_n$ over $E$.
\item $\Phi:=H/H^\circ$, where $H^\circ$ is the identity component of $H$.
\item $S:=\{v\in\Sigma_K\mid \rho\ {\rm unramified\ at\ }v\}$
\item$X:=\left\{(g_1,g_2)\in \GL_n\times\GL_n\mid \chi_{\rm char}(g_1)=\chi_{\rm char}(g_2) \right\}$, where $\chi_{\rm char}(g)$ is the characteristic polynomial of $g$. 
\item $\Sigma_X:=\{v\in S\mid \rho({\rm Frob}_v)\in X\}$
\end{itemize}

Suppose $H^\circ\subset X$, then we have the following statements:
\begin{enumerate}
\item[(i)] There is a finite extension $L/K$ such that $\rho_1|_{G_L}\cong\rho_2|_{G_L}$
\item[(ii)]  If $\rho_1$ is absolutely irreducible and $\overline{\rho_1(G_K)}^{\rm Zar}\subset \GL_n$ is connected, then there is a finite character $\chi:G_K\rightarrow E^*$ such that $\rho_2\cong\rho_1\otimes \chi$.
\end{enumerate}
\end{thm}

Before diving into the proof of the theorem, let us state and prove some of its related applications.

\begin{cor}\label{cor1}
Under the notation of Theorem \ref{thmb}, assume that the product representation
$\rho=(\rho_1,\rho_2)$ satisfies (\ref{CT}). If
$\udens(\Sigma_X)>1-1/|\Phi|$, then $\rho_1\cong\rho_2$.
\end{cor}
\begin{proof}
Apply Theorem \ref{thma}, $\delta(\Sigma_X)=|\Psi|/|\Phi|$ exists. The hypothesis
gives $|\Psi|>|\Phi|-1$, so $\Psi=\Phi$, i.e.\ $H\subset X$. Hence
$\rho_1(g)$ and $\rho_2(g)$ have the same characteristic polynomial for all
$g\in\GK$, and Theorem~\ref{BN} gives $\rho_1\cong\rho_2$.
\end{proof}

Corollary~\ref{cor1} is an analogue of Theorem~\ref{Ra1}: there the threshold $1-1/(2n^2)$ depends only on $n$, whereas
here it depends on $|\Phi|$ and requires (CT).

\begin{rem}\label{rem1}
One should compare the above corollary with the proof of Theorem 2(i) in \cite{Ra98} for the case $E$ is a local field of characteristic $0$. Rajan's strategy starts with the assumption that the upper density of $\Sigma_X$ is positive and $H_1=\overline{\rho_1(G_K)}^{\rm Zar}$ is connected, then there is a $\phi\in\Phi$ such that the corresponding connected component $J^\phi:=H^\phi\cap J$ lies in $X$, here $J$ is a maximal compact subgroup of $H(\mathbb{C})$. Using the fact that a maximal compact subgroup of $GL_n(\mathbb{C})$ is a group of unitary matrices, he can deduce that there is an element  $(I_n,j)\in J^\phi$ such that $j$ belongs to a group of unitary matrices $U_n(\mathbb{C})$. By trace counting and the fact that the only unitary matrix with trace $n$ is the identity matrix, $(I_n,j)$ is forced to be $(I_n,I_n)$, so $H^\phi=H^\circ\subset X$.   

In our situation where $E$ is of positive characteristic, we only have that any maximal compact subgroup of $GL_n(E)$ is conjugate to $\GL_n(\cO_E)$, hence Rajan's strategy does not work in positive characteristic case.

\end{rem}

\begin{cor}\label{cor2}
Assume that the conclusion of Theorem \ref{thmb}(ii) holds, i.e. $\rho_2\cong\rho_1\otimes \chi$ for some finite character $\chi:G_K\rightarrow E^*$, $\rho_1$ is absolutely irreducible and $H_1:=\overline{\rho_1(G_K)}^{\rm Zar}\subset \GL_n$ is connected. Suppose further that the local thinness property (LT) below holds for $\rho_1$:
\begin{equation}\tag{LT}\label{LT}
\textrm{Any proper Zariski-closed subset $Z$ of $H_1$ has Haar measure $\mu_{\rho_1(G_K)}(\rho_1(G_K)\cap Z)=0$}.
\end{equation}
Then the density $\delta(\Sigma_X)$ exists and equals $1/|\mathrm{Im}\,\chi|$. In particular, if density ${\delta}(\Sigma_X)>1/2$, then we have $\rho_1\cong\rho_2$. 
\end{cor}

\begin{proof}

If $\chi$ is trivial, then there is nothing to prove. Assume $\chi$ to be a nontrivial character, for $v\in S$, set $\zeta_v:=\chi({\rm Frob}_v)$. Then for $v\in \Sigma_X$, we have $$\chi_{\rm char}(\rho_1({\rm Frob}_v))=\chi_{\rm char}(\zeta_v\rho_1({\rm Frob}_v)).$$

Fix a  $\xi\in{\rm Im}(\chi)$ with $\xi\neq1$, and define $Z_\xi:=\left\{g\in H_1\mid \chi_{\rm char}(g)=\chi_{\rm char}(\xi g)\right\}$. $Z_\xi$ is a conjugation-stable proper Zariski-closed subset of $H_1$ because $I_n\not\in Z_\xi$. Then the property (\ref{LT}) implies $Z_\xi\cap \rho_1(G_K)$ has Haar measure 0 in $\rho_1(G_K)$. Since $Z_\xi(E)\cap\rho_1(\GK)$ is closed and
conjugation-stable, the claim in the proof of Theorem~\ref{thma} implies that $$\Sigma_\xi:=\{v\in S\mid \rho_1({\rm Frob}_v)\in Z_\xi(E)\}$$ has density $\delta(\Sigma_\xi)=0$.  Thus $\Sigma_{X,\xi}:=\{v\in \Sigma_X\mid \chi({\rm Frob}_v)=\xi\neq1\}$ has density equal to $0$, and we get $$\overline{\delta}(\Sigma_X)=\overline{\delta}\left(\Sigma_{X,1}\sqcup\sqcup_{1\neq\xi\in {\rm Im}\chi}\Sigma_{X,\xi}\right)=\overline{\delta}(\Sigma_{X,1})+\sum_{1\neq\xi\in{\rm Im}\chi}\overline{\delta}(\Sigma_{X,\xi})=\overline{\delta}(\Sigma_{X,1}),$$ 
here $\Sigma_{X,1}=\{v\in \Sigma_X\mid \chi({\rm Frob}_v)=1\}$. If there is a $v\in S$ such that $\chi(\Frob_v)=1$, then the assumption $\rho_2\cong\rho_1\otimes\chi$ implies $\rho_2(\Frob_v)$ and $\rho_1(\Frob_v)$ have the
same characteristic polynomial, so $v\in\Sigma_X$. Thus one can deduce that $\Sigma_{X,1}=\{ v\in S \mid \chi(\Frob_v)=1\}$, which has density
$1/|\mathrm{Im}\,\chi|$ by Theorem~\ref{thm:cheb}. Therefore, $\delta(\Sigma_X)$ exists and equal to
$\delta(\Sigma_{X,1})=1/|\mathrm{Im}\,\chi|$, this proves the
first assertion. The last assertion follows immediately since $1/|\mathrm{Im}\,\chi|\le1/2$
when $\chi\ne1$.

\end{proof}

Finally, we begin the proof of our main goal.
\begin{proof}[Proof of Theorem \ref{thmb}]

For $i\in\{1,2\}$, let $p_i:H\rightarrow GL_n$ be the projection onto the $i$-th coordinate. As $H^\circ\subset X$, we have 
$$\chi_{\rm char}(p_1(h))=\chi_{\rm char}(p_2(h))\ \textrm{for all }h\in H^\circ.$$
Now because $\rho_1$ and $\rho_2$ are semisimple  representations of $G_K$, we have that $p_1|_{H^\circ}$ and $p_2|_{H^\circ}$ are semisimple representations of $H^\circ$. Theorem \ref{BN} then implies 
\begin{equation}\tag{1}
p_1|_{H^\circ}\cong p_2|_{H^\circ}
\end{equation}
Let $\pi:G_K\xrightarrow{\rho}H(E)\twoheadrightarrow \Phi$ and set $L/K$ be the finite extension cut out by ${\rm ker}(\pi)$, then $\rho(G_L)\subset H^\circ(E)$ and restricting (1) further to $\rho(G_L)$ gives 
\begin{equation}\tag{2}
\rho_1|_{G_L}\cong\rho_2|_{G_L},
\end{equation}
this proves (i).

One may replace $L$ by its Galois closure in $G_K$ while (2) still holds. Thus, without loss of generality, we may assume $L/K$ is Galois.

Now assume $\rho_1$ is absolutely irreducible and $H_1:=\overline{\rho_1(G_K)}^{\rm Zar}\subset \GL_n$ is connected. Let $V_i$ be the $E$-space of $\rho_i$ for $i=1,2$, choose a $G_L$-equivariant isomorphism $f:V_1\rightarrow V_2$. 

\begin{claim*}
${\rm dim}_E({\rm Hom}_{G_L}(V_1,V_2))=1$
\end{claim*}
\begin{proof}[Proof of claim]

Choose a $G_L$-equivariant isomorphism $f:V_1\rightarrow V_2$. Consider the map
$$\eta:{\rm End}_{G_L}(V_1)\rightarrow {\rm Hom}_{G_L}(V_1,V_2),\ \ u\mapsto f\circ u.$$
One can easily see that $\eta$ is an $E$-isomorphism. Thus we turn to prove ${\rm dim}_E({\rm End}_{G_L}(V_1))=1$.

Since $H_1$ is connected, the finite index subgroup $\rho_1(G_L)\subset \rho_1(G_K)$ is still Zariski dense in $H_1$. Hence one can deduce that ${\rm End}_{G_L}(V_1)={\rm End}_{H_1}(V_1)={\rm End}_{G_K}(V_1)$, then Schur's lemma implies
$${\rm End}_{G_L}(V_1)={\rm End}_{G_K}(V_1)=E.$$
This completes the proof.
\end{proof}

Once the claim holds, for $\sigma\in G_K$ we define $f_\sigma:=\rho_2(\sigma)f\rho_1(\sigma)^{-1}$. As $G_L\lhd G_K$, for $h\in G_L$ one can deduce that
$$h\cdot f_\sigma=\rho_2(\sigma)\rho_2(\sigma^{-1}h\sigma)f\rho_1(\sigma)^{-1}=\rho_2(\sigma)f\rho_1(\sigma^{-1}h\sigma)\rho_1(\sigma)^{-1}=f_\sigma\cdot h.$$
Thus $f_\sigma\in {\rm Hom}_{G_L}(V_1,V_2)$. By the claim, ${\rm Hom}_{G_L}(V_1,V_2)$ is the $E$-space spanned by $f$, hence there is an element $\chi(\sigma)\in E^*$ such that $f_\sigma=\chi(\sigma)f$. This defines a map
$$\chi:G_K\rightarrow E^*.$$
For $\sigma_1,\sigma_2\in G_K$, we check that 
$$\chi(\sigma_1\sigma_2)f=f_{\sigma_1\sigma_2}=\rho_2(\sigma_1)f_{\sigma_2}\rho_1(\sigma_1)^{-1}=\chi(\sigma_2)f_{\sigma_1}=\chi(\sigma_1)\chi(\sigma_2)f.$$
Hence $\chi$ is a group homomorphism. Together with the fact that $\chi$ is continuous since $\sigma\mapsto f_\sigma$ is continuous, we know that $\chi$ is a character of $G_K$.  Moreover, from the property that $f_h=f$ for any $h\in G_L$, we know $\chi$ factors through the finite group ${\rm Gal}(L/K)$, so $\chi$ is a finite character.
Now from $f_\sigma=\chi(\sigma)f$ we get 
$$\rho_2(\sigma)f=f\chi(\sigma)\rho_1(\sigma)=f((\rho_1\otimes\chi)(\sigma)).$$
This means $f$ is a $G_K$-module isomorphism between $(V_1,\rho_1\otimes\chi)$ and $(V_2,\rho_2)$, hence $\rho_2\cong\rho_1\otimes\chi$.

\end{proof}

\subsection{Pairs of non-CM Drinfeld modules}
 
\begin{lem}\label{lem:thin}
\begin{enumerate}
\item[(a)] Let $U\subset\mathrm{GL}_n(E)$ be a compact open subgroup. If
$Z\subsetneq\mathrm{GL}_n$ is Zariski closed, then $\mu_U(U\cap Z(E))=0$.
\item[(b)] Let
$\mathbf H=\{(g_1,g_2)\in\mathrm{GL}_n\times\mathrm{GL}_n\mid\det g_1=\det g_2\}$,
and let $U\subset\mathbf H(E)$ be a compact open subgroup. If
$Z\subsetneq\mathbf H$ is Zariski closed, then $\mu_U(U\cap Z(E))=0$.
\end{enumerate}
\end{lem}
\begin{proof}
(a) Since $Z\neq\mathrm{GL}_n$ is Zariski closed, it lies in the zero set of
a non-zero polynomial $f$ in the $n^2$ matrix entries. Since $\det(U)$ is a compact subgroup of $E^\times$, its image under $|\cdot|$ is compact in $\mathbb{R}_{>0}$, hence $|\det u|=1$ for $u\in U$. The left multiplication by $u$ is a linear map of $\mathrm{M}_n(E)=E^{n^2}$ of determinant
$\det(u)^n$, it preserves the additive Haar measure $\lambda$ on $M_n(E)$, so
$\mu_U=\lambda|_U/\lambda(U)$ by uniqueness of Haar measure. It therefore
suffices to show that $\lambda(U\cap Z(E))=0$. The claim follows from the fact that the zero set of a
non-zero polynomial on $E^{n^2}$ has measure $0$ for the additive Haar
measure, which is proved in \cite[Lemma~3.1]{Yas17}. Note that Yasuda's result holds in any characteristic.

(b) Put $\delta(a,b)=\det b/\det a$ and $d_t=\mathrm{diag}(t,1,\dots,1)$, and
define the morphism
\[
\Phi\colon\mathrm{GL}_n\times\mathrm{GL}_n\to\mathbf H,\qquad
\Phi(a,b)=\bigl(a,\ b\,d_{\delta(a,b)}^{-1}\bigr);
\]
it is the identity on $\mathbf H$, hence surjective. Consider
\[
\psi\colon U\times\mathcal O_E^\times\to\mathrm{GL}_n(E)^2,\qquad
\psi((u_1,u_2),t)=u\cdot(1,d_t)=(u_1,u_2d_t).
\]
Since $\delta(\psi(u,t))=t$ and $\Phi(\psi(u,t))=u$, one checks easily that $\psi$ is a
homeomorphism from $U\times\mathcal O_E^\times$ onto
\[
W=\{(a,b)\in\mathrm{GL}_n(E)^2\mid\ \delta(a,b)\in\mathcal O_E^\times,\
\Phi(a,b)\in U\},
\]
with inverse $(a,b)\mapsto(\Phi(a,b),\delta(a,b))$. In particular, $W$ is
compact and open in $\mathrm{M}_n(E)^2=E^{2n^2}$.

Let $\lambda$ be the additive Haar measure on $E^{2n^2}$, normalised by
$\lambda(\mathcal O_E^{2n^2})=1$. Since $W$ is compact and open, we have
$0<\lambda(W)<\infty$. For $u=(u_1,u_2)\in U$, left multiplication by $u$ is a
linear map of $E^{2n^2}$ of determinant $\det(u_1)^n\det(u_2)^n$, which has
absolute value $1$ as in the proof of (a), hence it preserves $\lambda$.

For $\mu_U$-measurable set $A\subset U$ put $\nu(A)=\lambda\bigl(\psi(A\times\mathcal O_E^\times)\bigr)$.
Since $\psi$ is a homeomorphism onto $W$, $\nu$ is a measure on $U$ with
$\nu(U)=\lambda(W)<\infty$. Since $\psi(ua,t)=u\,\psi(a,t)$ and left
multiplication by $u\in U$ preserves $\lambda$, we have $\nu(uA)=\nu(A)$ for
all $u\in U$. By uniqueness of Haar measure, $\nu=\lambda(W)\,\mu_U$.
Now $Z':=\Phi^{-1}(Z)$ is a proper, Zariski-closed subset of
$\mathrm{GL}_n\times\mathrm{GL}_n$ because $\Phi$ is surjective, 
 and $Z\neq\mathbf H$.  Since $\Phi(\psi(u,t))=u$, we have $$W\cap Z'(E)=\psi\bigl((U\cap Z(E))\times\mathcal O_E^\times\bigr).$$ As in the proof of (a),
$Z'$ lies in the zero set of a non-zero polynomial in the $2n^2$ matrix
entries, so $\lambda(W\cap Z'(E))=0$ by \cite[Lemma~3.1]{Yas17}. Therefore
\[
\lambda(W)\,\mu_U\bigl(U\cap Z(E)\bigr)=\nu\bigl(U\cap Z(E)\bigr)
=\lambda\Bigl(\psi\bigl((U\cap Z(E))\times\mathcal O_E^\times\bigr)\Bigr)
=\lambda\bigl(W\cap Z'(E)\bigr)=0.
\]Thus  $\mu_U(U\cap Z(E))=0$.

\end{proof}
 
\begin{prop}\label{prop:pairs}
Let $\phi_1,\phi_2$ be non-CM Drinfeld modules of rank $r\ge2$ over $K$ of
generic characteristic, and let $\Sigma_X$ be the set of places at which
their Frobenius characteristic polynomials agree. Put
$\mathbf H=\{(g_1,g_2)\in\GL_r\times\GL_r\mid\det g_1=\det g_2\}$.
\begin{enumerate}
\item[(a)] If $\phi_1,\phi_2$ are isogenous over $\Kbar$, there is a character
$\kappa\colon\GK\to\Fq^\times$ of order $m\mid q-1$ with
$\rho_{\phi_2,\fl}\cong\rho_{\phi_1,\fl}\otimes\kappa$ for all $\fl$,
and $\delta(\Sigma_X)=1/m$. In particular
$\delta(\Sigma_X)=1$ if $\phi_1,\phi_2$ are $K$-isogenous, and
$\delta(\Sigma_X)\le1/2$ otherwise.
\item[(b)] Suppose $\phi_1,\phi_2$ are not isogenous over $\Kbar$, and that for
some prime $\fl$ the image of
$\rho=(\rho_{\phi_1,\fl},\rho_{\phi_2,\fl})$ contains an open subgroup of
$\mathbf H(\Fl)$. Then (\ref{CT}) holds and $\delta(\Sigma_X)=0$.
\end{enumerate}
\end{prop}
  
\begin{proof}
The Frobenius characteristic polynomials have coefficients in $A$, so
$\Sigma_X$ does not depend on $\fl$.
 
(a) A non-zero isogeny has non-zero constant term (compare the lowest
terms in $u\phi_{1,T}=\phi_{2,T}u$), so its kernel lies in $\Ksep$. The
leading coefficient $c$ of $u$ satisfies $c^{q^r-1}\in K$. Hence every isogeny
$\phi_1\to\phi_2$ over $\Kbar$ is defined over $\Ksep$. Consider the $A$-module $M=\Hom_{\Ksep}(\phi_1,\phi_2)$, where $a\in A$ acts on $u\in M$ via
$a\cdot u=\phi_{2,a}u=u\,\phi_{1,a}$. $M$ is torsion-free because $\Ksep\{\tau\}$
has no zero divisors. Fix $0\neq w\in M$, as $\ker w$ is a finite
$A$-module, $\ker w\subset\phi_1[a]$ for some $0\ne a\in A$, so there is
dual isogeny $\hat w\colon\phi_2\to\phi_1$ with $\hat w w=\phi_{1,a}$. The map
$M\to\End_{\Ksep}(\phi_1)=A$, $f\mapsto\hat wf$, is $A$-linear and
injective, so $M$ is isomorphic to a non-zero ideal of the principal ideal
domain $A$. Hence $M=A\,u$ for some $u$.
 
For $\sigma\in\GK$, applying $\sigma$ to coefficients maps $M$ to itself $A$-linearly and bijectively. So $\sigma$ acts
on $M\cong A$ by a unit $\kappa(\sigma)\in A^\times=\Fq^\times$, i.e. we have
$$\sigma(u)=\kappa(\sigma)\cdot u=\phi_{2,\kappa(\sigma)}u=\kappa(\sigma)u.$$ As
$\sigma$ fixes $\Fq$, the map $\kappa:G_K\rightarrow \Fq^\times$ is a homomorphism. Let $L$ be the finite separable extension of $K$ generated by coefficients of $u$, then $\kappa$ is trivial on the open
subgroup $G_L$, hence $\ker\kappa$ contains the open subgroup $G_L$, which means that $\kappa$ is continuous, of
some order $m\mid q-1$.

When $m=1$, this happens if and only if $u$ is defined over $K$, i.e. $\phi_1,\phi_2$ are $K$-isogenous. Thus $\rho_1\cong\rho_2$, and $\Sigma_X$ contains all unramified places, which implies $\delta(\Sigma_X)=1$.

When $m>1$, apply \cite[\S3.5, (3.5.3)]{P23}, $u$ induces an $F_\fl$-linear map
$u_\fl\colon V_\fl(\phi_1)\to V_\fl(\phi_2)$, which is an isomorphism since $\ker u$ is
finite, and both $V_\fl(\phi_1)$, $V_\fl(\phi_2)$ have $F_\fl$-dimension $r$. For $\sigma\in\GK$ and $x\in V_\fl(\phi_1)$ we have
$\sigma(u_\fl(x))=\sigma(u)(\sigma x)=\kappa(\sigma)\,u_\fl(\sigma x)$, i.e.\
$\rho_2(\sigma)u_\fl=\kappa(\sigma)\,u_\fl\rho_1(\sigma)=u_\fl(\rho_1\otimes\kappa)(\sigma)$. Thus
$\rho_2\cong\rho_1\otimes\kappa$.
The hypotheses of
Corollary~\ref{cor2} hold: $\rho_2\cong\rho_1\otimes\kappa$ with $\kappa$ of
finite order; $\rho_1$ is absolutely irreducible and $H_1=\GL_r$ is
connected because $\rho_1(\GK)$ is open in $\GL_r(\Fl)$ by \cite[Thm.~0.1]{PR09}; and (\ref{LT}) holds for $\rho_1$ by applying Lemma~\ref{lem:thin}(a) to
the compact open subgroup $\rho_1(\GK)$ of $\GL_r(\Fl)$. Since
$|\mathrm{Im}\,\kappa|=m$,  Corollary~\ref{cor2} gives
$\delta(\Sigma_X)=1/m$.
 
(b) Let $\Gamma=\rho(\GK)$, let $H$ be its Zariski closure, and let $\Omega$ be
the open subgroup of $\mathbf H(\Fl)$ contained in $\Gamma$. By
Theorem~\ref{thma} it suffices to show that $H^\circ=\mathbf H$, that (\ref{CT}) holds,
and that $\Psi=\emptyset$. Then $\delta(\Sigma_X)=|\Psi|/|\Phi|=0$.
 
\medskip\noindent{\emph{Step 1: $H^\circ=\mathbf H$.}}\par\smallskip By the Weil pairing \cite{vdH04},
$\det\rho_{\phi_i,\fl}=\rho_{\psi_i,\fl}$, where
$\psi_{i,T}=T+(-1)^{r-1}\Delta_i\tau$ and $\Delta_i$ is the leading
coefficient of $\phi_{i,T}$. Conjugating by $c$ with $c^{q-1}=\Delta_2/\Delta_1$
turns $\psi_1$ into $\psi_2$, so
$\psi_1\cong\psi_2$ over the finite extension
$L=K\bigl((\Delta_2/\Delta_1)^{1/(q-1)}\bigr)$. We have that $\det\rho_{\phi_1,\fl}=\det\rho_{\phi_2,\fl}$ on $G_L$, i.e.\
$\rho(G_L)\subset\mathbf H(\Fl)$. Since $\rho(G_L)$ has finite index in
$\Gamma$, its Zariski closure $\overline{\rho(G_L)}^{\rm Zar}$ is a closed
subgroup of finite index in $H$, hence contains the identity component
$H^\circ$. As $\mathbf H$ is Zariski closed, we obtain
$H^\circ\subset\overline{\rho(G_L)}^{\rm Zar}\subset\mathbf H$. Conversely, $\Omega$ is Zariski dense in $\mathbf H$ since otherwise $\Omega\subset Z(\Fl)$
for some Zariski-closed $Z\subsetneq\mathbf H$, so $\mu_\Omega(\Omega\cap Z(\Fl))=1$,
contradicting Lemma~\ref{lem:thin}(b). Since
$\Omega\subset\Gamma$, taking Zariski closures gives
$\mathbf H\subset H$. We observe that there is an isomorphism of varieties $\mathbf H\rightarrow \GL_r\times\SL_r$ via
$(g_1,g_2)\mapsto(g_1,g_1^{-1}g_2)$, so the group $\mathbf H$ is connected. This implies $\mathbf H\subset H^\circ$.
 
\medskip\noindent{\emph{Step 2: condition (\ref{CT}).}}\par\smallskip Put $\Gamma^\circ:=\Gamma\cap\mathbf H(\Fl)$. Since $\mathbf H=H^\circ$, the
quotient $\Gamma/\Gamma^\circ$ embeds into $\Phi$, so $\Gamma^\circ$ has finite
index in $\Gamma$. Moreover, $\Gamma^\circ$ is closed in the compact group
$\Gamma$, hence compact. As $\Gamma^\circ$ contains the open subgroup $\Omega$ of
$\mathbf H(\Fl)$, it is a compact open subgroup of $\mathbf H(\Fl)$. 

Let $\varphi\in\Phi$ with $H^\varphi\not\subset X$. Since
$\Gamma$ is Zariski dense in $H$, we have $\Gamma\cap H^\varphi(\Fl)\neq\emptyset$. Choose $\gamma\in\Gamma\cap H^\varphi(\Fl)$, then $H^\varphi=\gamma H^\circ=\gamma\mathbf H$.
Since $\gamma\in\Gamma$, we get
$\Gamma\cap H^\varphi(\Fl)=\gamma\bigl(\Gamma\cap\mathbf H(\Fl)\bigr)=\gamma\Gamma^\circ$.

Note that $\gamma^{-1}(X\cap H^\varphi)$ is a proper Zariski-closed subset of
$\gamma^{-1}H^\varphi=\mathbf H$, because $X\cap H^\varphi\neq H^\varphi$. As
$\gamma\Gamma^\circ\subset H^\varphi(\Fl)$, we have
$\gamma\Gamma^\circ\cap X(\Fl)=\gamma\bigl(\Gamma^\circ\cap\gamma^{-1}(X\cap H^\varphi)(\Fl)\bigr)$,
and the restriction of $\mu_\Gamma$ to the open subgroup $\Gamma^\circ$ is
$[\Gamma:\Gamma^\circ]^{-1}\mu_{\Gamma^\circ}$. Applying
Lemma~\ref{lem:thin}(b) with $\Fl$, $\Gamma^\circ$ and
$\gamma^{-1}(X\cap H^\varphi)$ in place of $E$, $U$ and $Z$, we get
\[
\mu_\Gamma\bigl(\gamma\Gamma^\circ\cap X(\Fl)\bigr)
=[\Gamma:\Gamma^\circ]^{-1}\,
\mu_{\Gamma^\circ}\bigl(\Gamma^\circ\cap\gamma^{-1}(X\cap H^\varphi)(\Fl)\bigr)=0.
\]

\medskip\noindent{\emph{Step 3: $\Psi=\emptyset$.}}\par\smallskip 
By Step~1, the components of $H$ are the cosets $(h_1,h_2)\mathbf H$ with
$(h_1,h_2)\in H$, so it suffices to show that none of them is contained in
$X$. Suppose $(h_1,h_2)\mathbf H\subset X$. For
$g\in\SL_r(\Fl)$, we have $(g,1)\in\mathbf H$, so $(h_1g,h_2)\in X$, and
comparing traces gives $\tr(h_1g)=\tr(h_2)$. Taking $g=1$ shows
$\tr(h_1)=\tr(h_2)$, hence $\tr\bigl(h_1(g-1)\bigr)=0$ for all
$g\in\SL_r(\Fl)$. For $g=1+E_{ij}$ with $i\ne j$ this gives
$(h_1)_{ji}=0$, so $h_1$ is diagonal. Taking $g=\operatorname{diag}(a,a^{-1},1,\dots,1)$ with $a\in F_\fl^\times$, 
it gives $(h_1)_{11}(a-1)+(h_1)_{22}(a^{-1}-1)=0$. Multiplying the equation by $a$, we see that any
$a\in F_\fl^\times$ is a root of the quadratic equation
\[
P(x)=(h_1)_{11}x^2-\bigl((h_1)_{11}+(h_1)_{22}\bigr)x+(h_1)_{22}\in\Fl[x]
\]
Since $\Fl$ is infinite, we must have $(h_1)_{11}=(h_1)_{22}=0$. The
same argument for every pair of indices gives $h_1=0$, contradicting
$h_1\in\GL_r$. 
\end{proof}

 \begin{cor}\label{cor:pairs}
Let $\phi_1,\phi_2$ be as in Proposition~\ref{prop:pairs}, and assume that
they are isogenous over $\Kbar$, or that the openness hypothesis of
Proposition~\ref{prop:pairs}(b) holds. If $\delta(\Sigma_X)>1/2$, then $\phi_1$ and $\phi_2$ are
$K$-isogenous. For $q=2$ it suffices that $\delta(\Sigma_X)>0$.
\end{cor}

\begin{proof}
In case (a), we have $\delta(\Sigma_X)=1/m>1/2$ implies $m=1$. In case (b), we have
$\delta(\Sigma_X)=0$, which cannot happen. For $q=2$, $\Fq^\times$ is trivial. Thus $\udens(\Sigma_X)>0$ forces $m=1$ in case (a).
\end{proof}
 
\begin{rem}\label{uncon}
(1) For $q$ odd, let $c\in K^\times$ be a non-square and let $\phi_2$ be
the quadratic twist $\sqrt c^{\,-1}\phi_1\sqrt c$. Part (a) with $m=2$
gives $\delta(\Sigma_X)=1/2$ while $\phi_1,\phi_2$ are not $K$-isogenous.
So the bound $1/2$ in Corollary~\ref{cor2} and Corollary~\ref{cor:pairs} is
sharp.
 
(2) The openness hypothesis in (b) is the function field analogue of
Serre's open image theorem for products of non-isogenous elliptic curves,
and it holds by the recent work of Duan and Fang \cite[Thm.~1.2]{DF26}.
In the terminology of \cite[Definition 2.1]{DF26}, for non-CM Drinfeld modules of the same rank, a Frobenius twist
$\phi_1^{(d)}$ with $d\ge1$ has rank $rp^d\neq r$, so the hypothesis of
\cite[Thm.~1.2]{DF26} reduces to $\phi_1,\phi_2$ not being isogenous over
$\Kbar$. Since $\End_{\Kbar}(\phi_i)=A$, the group $U_\fl$ of
\cite[Thm.~1.2]{DF26} is
\[
U_\fl=\{(g_1,g_2)\in\GL_r(\Al)\times\GL_r(\Al)\mid\det g_1=\det g_2\}
=\mathbf H(\Fl)\cap\GL_r(\Al)^2 ,
\]
where $\GL_r(\Al)$ is identified with $\Aut_{\Al}(T_\fl(\phi_i))$. As
$\GL_r(\Al)$ is open in $\GL_r(\Fl)$, we know that $U_\fl$ is a compact open subgroup of
$\mathbf H(\Fl)$. The theorem asserts that $\Gamma$ and $U_\fl$ are
commensurable, so $\Gamma\cap U_\fl$ is a closed subgroup in $U_\fl$ of finite index, it is open in $U_\fl$, hence open in $\mathbf H(\Fl)$. Thus Proposition~\ref{prop:pairs}(b) and Corollary~\ref{cor:pairs} hold
unconditionally.
 
(3) Combining (a) and (b) in Proposition~\ref{prop:pairs}: if $\phi_1,\phi_2$ are non-CM Drinfeld modules
of rank $r\ge2$ over $K$ with $\udens(\Sigma_X)>0$, then they are
isogenous over $\Kbar$. We have that $\rho_{\phi_2,\fl}\cong\rho_{\phi_1,\fl}\otimes\kappa$
for a character $\kappa\colon\GK\to\Fq^\times$ of order $m\mid q-1$, and
$\delta(\Sigma_X)=1/m$. This is the function field analogue of Rajan's
result for elliptic curves (see Corollary~\ref{Ra2}), with
$A^\times=\Fq^\times$ in place of $\mathbb Z^\times=\{\pm1\}$ and ``twists by
characters of order dividing $q-1$" replace ``quadratic twists". Rajan deduces
his corollary from \cite[Thm.~2]{Ra98}, whose proof uses maximal compact
subgroups of $H(\mathbb C)$ as explained in Remark~\ref{rem1}, this argument does
not carry over to positive characteristic. Here it is replaced by an
explicit determination of $H$, through Pink's open image theorem in case
(a) and through \cite{DF26} in case (b). Note that $\Sigma_X$ is defined by
equality of Frobenius characteristic polynomials rather than of Frobenius traces.
\end{rem}

\section*{Statements and Declarations}
\subsection*{Data Availability Statements}

We do not analyse or generate any datasets, because our work proceeds within a theoretical and mathematical approach.

\subsection*{Conflict of Interest statement}
The author has no competing interests to declare that are relevant to the content of this article.

\subsection*{Funding Declaration}
The author was partially supported by Center for Advanced Mathematical Sciences Fellowship for the 2026--2027 cycle.

\bibliographystyle{alpha}
\bibliography{Frobenius_trace_DM_rev1.bib}

\end{document}